\documentclass[12pt]{amsart}

\usepackage{amsmath}
\usepackage{amsfonts}
\usepackage{amssymb}
\usepackage{hyperref}

\DeclareMathOperator{\tr}{tr}

\DeclareMathOperator{\R}{\mathbb{R}}
\DeclareMathOperator{\T}{\mathbb{T}}

\DeclareMathOperator{\Hilbert}{\mathcal{H}}

\DeclareMathOperator{\ltrpv}{\lvert\mspace{-1.2mu}\lvert\mspace{-1.2mu}\lvert}

\newtheorem{theorem}{Theorem}[section]

\newtheorem{corollary}[theorem]{Corollary}

\theoremstyle{definition}

\theoremstyle{remark}
\newtheorem{remark}[theorem]{Remark}

\numberwithin{equation}{section}

\begin{document}


\title[
Generic factorization
]
{Noncommutative Khintchine\\
and Paley inequalities\\
via generic factorization}
\author{John J.F. Fournier}
\address{Department of Mathematics\\
University of British Columbia\\
1984 Mathematics Road\\
Vancouver BC\\
Canada V6T 1Z2}
\email{fournier@math.ubc.ca}
\thanks{Announced at the 6th Conference on Function Spaces in May 2010.}
\thanks{Research partially supported by NSERC grant 4822.}


\begin{abstract}
We reprove an inequality for Rademacher series with coefficients in the Schatten class~$S_1$. Our method 
yields
the same estimate for coefficients after
suitable
gaps in~$S_1$-valued trigonometric series; this
was known for scalar-valued functions.
A very similar method
gives a new proof of
the extension 
to~$S_1$-valued~$H^1$ functions
of
Paley's theorem 
about
lacunary coefficients.
\end{abstract}

\maketitle

\markleft{
John J.F. Fournier
}

\section{Introduction}\label{sec:intro}

Given
a function~$f$ in~$L^1((-\pi, \pi])$,
form its
Fourier coefficients
\begin{equation}
\label{eq:Fourier}
\hat f(n) = \frac{1}{2\pi}\int_{-\pi}^\pi
f(t)e^{-int}\,dt.
\end{equation}
Identify the interval~$(-\pi, \pi]$ with the unit circle group~$\T$ in the usual way.
Extend
formula~\eqref{eq:Fourier}
to Bochner integrable functions~$f$ mapping~$\T$ into
the Schatten classes~$S_p$ with~$1 \le p < \infty$.
Let
\[
\|f\|_{L^p(\T;S_p)} = \left\{\frac{1}{2\pi}\int_{-\pi}^\pi
\left(\|f(t)\|_{S_p}\right)^p\,dt\right\}^{1/p}.
\]
In Section~\ref{sec:Properties}, we 
discuss the properties of~$S_p$ and~$L^p(
\T;
S_p)$ that we use.

Let~$(k_j)_{j=0}^\infty$
be a sequence of nonnegative integers
for which
\begin{equation}
k_{j+1} > 2k_j
\end{equation}
for all $j$. Denote the range of such a sequence by~$K$,
and call~$K$ \textit{strongly lacunary.}
In Theorem~\ref{th:Paley} below, we present two cases where, if~$f \in L^1(\T; S_1)$ and if~$\hat f$ vanishes on a suitable subset of the complement of such a set~$K$, then the restriction of~$\hat f$ to~$K$ has special properties.  This is known for scalar-valued functions, as is one~\cite{LP} of the cases for functions with values in~$S_1$, but our 
proof
is 
new in that context.

We also consider functions in~$L^1([0, 1); S_1)$ whose Walsh coefficients, in the Paley ordering, vanish except at the powers of~$2$; see Section~\ref{sec:two-step} for more details. The Walsh series for such functions reduce to Rademacher series
$
\sum_{j=0}^\infty d_j r_j(\cdot)
$
with~$S_1$-valued coefficients.
We give a new proof that these coefficients also have special properties.

Those involve the following 
norm
on some sequences,~$(c_j)$ say, of compact operators on a Hilbert space,~$\Hilbert$ say.
Let
\begin{equation}
\label{eq:ColumnNorm}
\|(c_j)\|_{\mathcal{C}_E}
=
\left
\|\sqrt{\sum_j c_j^*c_j}
\right
\|_{S_1}
= \tr\sqrt{\sum_j c_j^*c_j}.
\end{equation}
This is
the norm in~$S_1(\ell^2(\Hilbert))$ of any operator-valued
square 
matrix
in which one column
is the sequence~$(c_j)$ and the other columns are trivial. 
Note that
\begin{equation}
\label{eq:RowNorm}
\|(c_j^*)\|_{\mathcal{C}_E}
=
\left
\|\sqrt{\sum_j c_jc_j^*}
\right
\|_{S_1}
= \tr\sqrt{\sum_j c_jc_j^*}.
\end{equation}
Finally, follow~\cite{LP} and let
\begin{equation}
\label{eq:SplittingNorm}
\ltrpv
(c_j)
\rvert\mspace{-1.2mu}\rvert\mspace{-1.2mu}\rvert = \inf\left\{\|(a_j)\|_{\mathcal{C}_E}
+ \|(b_j^*)\|_{\mathcal{C}_E}:
(c_j) = (a_j) + (b_j)\right\}.
\end{equation}
In~\cite{HaaMus} this functional is denoted instead by~$\ltrpv
\cdot
\rvert\mspace{-1.2mu}\rvert\mspace{-1.2mu}\rvert^*$,
and a dual norm is denoted by~$\ltrpv
\cdot
\rvert\mspace{-1.2mu}\rvert\mspace{-1.2mu}\rvert$.

We prove
the assertions below,
in which~$C$ is an absolute constant,
in Section~\ref{sec:two-step}.
%

\begin{theorem}
\label{th:Khintchine}
If a Rademacher series with coefficients~$(d_j)$ represents
a function~$f$ in~$L^1([0, 1); S_1)$, then
\begin{equation}
\label{eq:Khintchine}
\ltrpv(d_j)
\rvert\mspace{-1.2mu}\rvert\mspace{-1.2mu}\rvert
\le C\|f\|_{L^1([0, 1);S_1)}.
\end{equation}
\end{theorem}

\begin{theorem}
\label{th:Paley}
Let~$K$ be a strongly lacunary set,
and let~$f \in L^1(\T; S_1)$.
Then the estimate
\begin{equation}
\label{eq:Paley}
\ltrpv(\hat f(k_j))
\rvert\mspace{-1.2mu}\rvert\mspace{-1.2mu}\rvert
\le C\|f\|_{L^1(\T, S_1)}
\end{equation}
holds in 
each of
the following cases:
\begin{enumerate}
\item{}
\label{it:Paley}
$\hat f$ vanishes
on the set of negative integers.
\item{}
\label{it:Complementary}
$\hat f$ vanishes
at all positive integers in the complement of~$K$.
\end{enumerate}
\end{theorem}

\begin{corollary}
\label{th:Steinhaus}
The same estimate
holds when~$\hat f$ vanishes off~$K$.
\end{corollary}

The corollary is a counterpart for trigonometric series of Theorem~\ref{th:Khintchine}.
In~\cite{LP}, Lust and Pisier proved Case~\ref{it:Paley} 
of Theorem~\ref{th:Paley}
first, and 
deduced Theorem~\ref{th:Khintchine} from the corollary.
Case~\ref{it:Paley} is an extension
of Paley's theorem about lacunary
coefficients of scalar-valued~$H^1$~functions~\cite{REACH1}. The proof in~\cite{LP}, like Paley's, used a suitable factorization
of~$H^1$~functions as products of~$H^2$~functions.
That analytic factorization does not apply as readily in
Case~\ref{it:Complementary}, and it is not available in Theorem~\ref{th:Khintchine}.
Instead, we give direct proofs of Theorem~\ref{th:Khintchine} and both parts of Theorem~\ref{th:Paley}
using the generic factorization method introduced in~\cite[Section~2]{FouPac}
and modified in \cite[Section~2]{FouMiss}.
See Remark~\ref{rm:Factorization} below for further comparison of methods.

I thank Christian Le Merdy and Fedor Sukochev for pointing our an error in a earlier version of this paper. For more details, see Remark~\ref{rm:Equivalence}.

\begin{remark}
The methods used here work when~$C=2$.
Dual methods in~\cite{HaaMus} show that Corollary~\ref{th:Steinhaus} and Theorem~\ref{th:Khintchine} hold with~$C$ equal to
~$\sqrt 2$ and~$\sqrt 3$ respectively,  and that~$\sqrt 2$ is the best constant in the
corollary.
For scalar valued functions, dual methods in~\cite{FouArk} and~\cite{Clu} yield
the two cases of
Theorem~\ref{th:Paley} with constants~$\sqrt2$ and~$\sqrt e$ respectively; again, the former is best possible
in that context.
It is not known to what extent 
those
dual methods
extend to operator-valued functions.
As indicated in Remark~\ref{rm:Refine}
below,
further analysis of the methods used here supports the possibility of such extensions.
\end{remark}


\begin{remark}
Case~\ref{it:Complementary} of Theorem~\ref{th:Paley}
for scalar-valued functions
was rediscovered several times.
That instance
follows
easily
from another theorem
of Paley~\cite{REACM}, but this
may not have been noticed
until~\cite[Theorem~10]{FouPac}
and~\cite{Goes}.
Meanwhile, an equivalent dual construction 
had been found~\cite{Clu} 
by Clunie, and
Meyer~\cite[pp.~532-533]{Mey} had 
had
used other methods to prove
pointwise 
estimates that imply inequality~\eqref{eq:Paley}
in that instance.
Related proofs of 
Case~\ref{it:Complementary} 
for 
such
functions
appeared in~\cite{Vin},~\cite[Theorem~11]{FouPac}
and~\cite{Leb}.
The hypotheses about~$K$ and~$f$ were significantly weakened 
in~\cite{VinK}.
\end{remark}

\section{Properties of
these operator spaces
}
\label{sec:Properties}

See~\cite{Sim} for much more
about the spaces~$S_p$, which are often denoted by~$c_p$ or~$C_p$.
The members of~$S_1$ are
the compact operators, on 
the
fixed Hilbert space~$\Hilbert$, whose sequence of singular values belongs to~$\ell^1$.
Define~$\|A\|_{S_1}$ to be the sum of that sequence, that is the sum of the eigenvalues of~$|A| := \sqrt{A^*A}$.
Also denote that sum is by~$\tr(|A|)$, thereby defining 
the functional~$\tr$
on the set of positive operators in~$S_1$.
It
extends to become a linear map from~$S_1$
into the complex numbers.

For each positive real number~$p$, the space~$S_p$ consists of all operators~$A$ for which~$|A|^p \in S_1$; then
\[
\|A\|_{S_p} :=
\left\{\tr\left[\left(\sqrt{A^*A}\right)^p\right]\right\}^{1/p}
= \left\{\left\|\left(\sqrt{A^*A}\right)^p\right\|_{S_1}\right\}^{1/p}.
\]
There is a counterpart of H\"older's inequality, stating that
\begin{equation}
\label{eq:Hoelder}
\|AB\|_{S_r} \le \|A\|_{S_p}\|B\|_{S_q}
\quad\text{when}\quad
\frac{1}{r} = \frac{1}{p} + \frac{1}{q}.
\end{equation}
Conversely, for such indices,
operators the unit ball of~$S_r$ factor as
products of operators in the unit balls of~$S_p$ and~$S_q$.
There also are such factorizations in the form~$B^*A$, since~$\|B^*\|_{S_q} = \|B\|_{S_q}$.
%

The sets~$S_p$ are Banach spaces when~$1 \le p < \infty$. Moreover,~$S_2$ is a Hilbert space, with inner product~$
\langle A, B\rangle := \tr(B^*A)$.
It follows that~$L^2(\T; S_2)$ is also a Hilbert space with
inner product
\[
\langle f, g\rangle :=
\frac{1}{2\pi}\int_{-\pi}^\pi
\tr[g(t)^*f(t)]\,dt.
\]
This is 
also 
equal to
the trace of
\begin{equation}
\label{eq:PartialProduct}
\langle f, g\rangle_p := \frac{1}{2\pi}\int_{-\pi}^\pi
g(t)^*f(t)\,dt.
\end{equation}
Call 
the
operator-valued 
expression 
above
a \emph{partial inner product.}

Finally~$L^1(\T; S_1)$ is a Banach space, and
a function belongs to its unit ball 
if and only that function is the product of two functions in the unit ball of~$L^2(\T; S_2)$.
Call that a \emph{generic} factorization.
The  proof of Case~\ref{it:Paley} in~\cite{LP} also
uses such factors,
but adds the requirement~\cite{Sar, Muh, HaaPi}
that their Fourier coefficients vanish at all negative integers. This is called \emph{analytic factorization.}

\begin{remark}
As noted in~\cite[p. 250]{LP}, a weaker form of analytic factorization suffices for that proof of Case~\ref{it:Paley}. In Remark~\ref{rm:WeakFactor}, we explain how our method works with a weaker form of generic factorization.
\end{remark}


\section{Two orthogonality steps}\label{sec:two-step}

Recall that the Rademacher function~$r_0$ has period~$1$ on the real line~$\R$, and takes the values~$1$ and~$-1$ in the intervals~$[0, 1/2)$ and~$[1/2, 1)$ respectively; then~$r_j(t) := r_0(2^j t)$ for all~$t$ in~$\R$
and each positive integer~$j$.
The Walsh functions~$(w_n)_{n=0}^\infty$ are
the
products of 
finitely-many
Rademacher functions, including the empty product~$1$. We use the Paley enumeration of this 
system, where~$w_n$ is the product of the 
distinct 
functions~$r_{j_m}$ for which~$n = \sum_m 2^{j_m}$;
in particular,~$w_0 = 1$ and~$w_{2^j} = r_j$. Our method transfers easily to other standard enumerations.

\begin{proof}[Proof of Theorem~\ref{th:Khintchine}]
Rescale to make~$\|f\|_{L^1([0, 1); S_1)} = 1.$ Then factor~$f$ as~$h^*g$ where~$g$ and~$h$ belong to the unit ball of the Hilbert space~$L^2([0, 1); S_2)$. Rewrite the 
Walsh coefficients of~$f$ in the form
\begin{equation}
\label{eq:WalshCoefficient}
\hat f(w_n)
= \int_0^1 w_n(t)h^*(t)g(t)\,dt;
\end{equation}
in this context, use the notation~$\langle g, w_n h\rangle_p$ for the integral above.
The
hypothesis in the theorem is
that this partial inner product vanishes
when~$n$ is not a power of~$2$.

Let~$A_j$ be the unitary operator on~$
L^2([0, 1); S_2)$ that multiplies each function by~$w_{2^j}$. Matters reduce to splitting the sequence
\[
(\hat f(w_{2^j})) =
(\langle g, A_jh\rangle_p)
\]
as a sum of two sequences~$(a_j)$ and~$(b_j)$ for which
\begin{equation}
\label{eq:SplitNorms}
\|(a_j)\|_{\mathcal{C}_E} \le \frac{C}{2}
\quad\text{and}\quad
\|(b^*_j)\|_{\mathcal{C}_E} \le \frac{C}{2}.
\end{equation}
This will be done using
nested 
closed
subspaces
of~$L^2([0, 1);S_2)$
and
the
orthogonal projections onto them.

Denote the 
set
of functions orthogonal to 
a 
closed
subspace,~$M$ say, 
of~$L^2([0, 1);S_2)$
by~$M^\perp$. Since~$\langle u, v\rangle = 0$ when~$u \in M$ and~$v \in M^\perp$, the trace of~$\langle u, v\rangle_p$ vanishes in that case.

For some subspaces~$M$, the stronger condition that~$\langle u, v\rangle_p= 0$ holds when~$u \in M$ and~$v \in M^\perp$. We claim that this happens 
if~$ub \in M$ whenever~$u \in M$ 
and~$b$ belongs to the space~$B(\Hilbert)$ of bounded operators
on~$\Hilbert$.
Indeed, suppose that~$M$ has the latter property,  and observe
that
\begin{equation}
\label{eq:Module}
\langle ub, v\rangle_p =
(\langle u, v\rangle_p)b
\end{equation}
for all functions~$u$ and~$v$ in~$L^2([0, 1);S_2)$
and all bounded operators~$b$.
If~$u \in M$ and~$v \in M^\perp$, then~$v \perp ub$ 
since~$MB(\Hilbert) \subset M$;
that is,
\[
0 = \langle ub, v\rangle = \tr \langle ub, v\rangle_p.
\]
By equation~\eqref{eq:Module},~$\tr[( \langle u, v\rangle_p)b]$
then vanishes for all~$b$ in~$B(\Hilbert)$.
Therefore~$\langle u, v\rangle_p = 0$,
as claimed.

In that situation, denote the orthogonal projection with range equal to the closed subspace~$M$ by~$Q$, and denote the orthogonal projection with range~$M^\perp$ by~$Q^\perp$. Split any two members~$F$ and~$G$ of~$L^2([0, 1);S_2)$ as~$QF + Q^\perp F$ and~$QG + Q^\perp G$.
Since~$\langle Q^\perp F, QG\rangle_p$ and~$\langle QF, Q^\perp G\rangle_p$ both vanish, 
\begin{equation}
\label{eq:Projections}
\langle F, QG\rangle_p
= \langle QF, QG\rangle_p
= \langle QF, G\rangle_p.
\end{equation}

Let~$M_j$ be the smallest closed subspace of~$L^2([0, 1); S_2)$
that contains all
the products~$w_n hb$ in which~$0 < n < 2^{j+1}$ and~$b \in B(\Hilbert)$;
those factors~$w_n$ are
the
nonempty products of 
distinct
Rademacher functions~$r_{j'}$ with~$j' \le j$. Clearly,
\begin{equation}
\label{eq:M'sNest}
M_0 \subset M_1 \subset \cdots \subset M_J \subset \cdots.
\end{equation}
Then~$A_jM_j$ is 
the smallest closed subspace containing all
the products~$w_n hb$ 
with~$n \ne 2^j$
and~$0 \le n < 2^{j+1}$,
and with~$b \in B(\Hilbert)$;
these factors~$w_n$ are the products of 
distinct
Rademacher functions~$r_{j'}$ with~$j' \le j$
except for the singleton product that gives~$r_{j}$.
Again,
\begin{equation}
\label{eq:AM'sNest}
A_0M_0 \subset A_1M_1 \subset \cdots \subset A_jM_j \subset \cdots.
\end{equation}

Moreover,~$A_jh \in A_{j+1}M_{j+1}$. 
Finally,~$A_{j+1}M_j$ is 
the smallest closed subspace containing all
the products~$w_n hb$ with~$2^{j+1} < n < 2^{j+2}$ and~$b$ in~$B(\Hilbert)$. 
It is assumed in Theorem~\ref{th:Khintchine}
that~$\hat f(w_n)$ for these indices~$n$,
that is~$\langle g, w_{n} h\rangle_p = 0$. 
Equation~\eqref{eq:Module} then 
makes~$\langle g, w_nhb\rangle_p = 0$
for all~$b$ in~$B(\Hilbert)$.
Hence~$\langle g, v\rangle_p = 0$
for all functions~$v$ in~$A_{j+1}M_j$. 
Write this as~$g\perp_p A_{j+1}M_j$.

Denote the orthogonal projection onto~$A_jM_j$ by~$Q_j$.
The subspaces~$M_j$ were chosen so that~$M_jB(\Hilbert) \subset M_j$,
and their images~$A_jM_j$ also have this property.
Use the fact that~$A_j h \in A_{j+1}M_{j+1}$, and 
apply equation~\eqref{eq:Projections} with~$M = A_{j+1}M_{j+1}$ to write
\begin{gather}
\label{eq:SelfAdjoint}
\langle g, A_jh\rangle_p
= \langle g, Q_{j+1}A_jh\rangle_p
= \langle Q_{j+1}g, A_jh\rangle_p\\
\label{eq:Splitting}
=\quad
\langle Q_jg, A_jh\rangle_p
+
\langle(Q_{j+1}-Q_j)g, A_jh\rangle_p
= a_j + b_j
\quad\textnormal{say.}
\end{gather}

Let~$g_j = (Q_{j+1} - Q_j)g$.
Since the orthogonal projections~$Q_j$ nest,
\[
\sum_j
\|g_j\|_{L^2([0, 1); S_2)}^2
\le
\|g\|_{L^2([0, 1); S_2)}^2
= 1.
\]
Now argue as in~\cite[p.~250]{LP}.
Write~$b_j^*$ as~$\int_0^1 w_{2^j}(t)g_j^*(t)h(t)\,dt$, and regard this as the average of the operators~$w_{2^j}(t)g_j^*(t)h(t)$.
Since~$\|\cdot\|_{\mathcal{C}_E}$ is a norm,
\[
\|(b_j^*)\|_{\mathcal{C}_E}
\le \int_0^1 \left\|\left(w_{2^j}(t)g_j^*(t)h(t)\right)\right\|_{\mathcal{C}_E}\,dt
= \int_0^1\left \|\left(g_j^*(t)h(t)\right)\right\|_{\mathcal{C}_E}\,dt.
\]
Fix~$t$, and expand the last inner norm above as
\[
\left\|
\sqrt{
h(t)^*\left[\sum_j g_j(t)g_j^*(t)\right]h(t)
}
\,\,\right\|_{S_1}.
\]
Let~$G(t)$ be the
operator~$\sum_j g_j(t)g_j^*(t)$.
The quantity above is equal to the square root of~$\|h(t)^*G(t)h(t)\|_{S_{1/2}}$.
%
By the H\"older inequality~\eqref{eq:Hoelder},
\[
\|h(t)^*G(t)h(t)\|_{S_{1/2}}
\le
\|h(t)^*\|_{S_2}
\|G(t)h(t)\|_{S_{2/3}},
\]
\[
\text{and}\quad
\|G(t)h(t)\|_{S_{2/3}}
\le
\|G(t)\|_{S_1}\|\|h(t)\|_{S_2}.
\]
Hence
\[
\|(b_j^*)\|_{\mathcal{C}_E}
\le \int_0^1
\sqrt{
\|h(t)^*\|_{S_2}\|G(t)\|_{S_1}\|\|h(t)\|_{S_2}
}
\,dt
\]
\[
= \int_0^1
\|h(t)\|_{S_2}\sqrt{\|G(t)\|_{S_1}}
\,dt.
\]
Cauchy-Schwarz then gives the upper bound
\[
\|h\|_{L^2([0,1);S_2)}
\sqrt{\|G\|_{L^1([0, 1);S_1)}}\quad.
\]
The first factor is equal to~$1$.
The square of the second factor is 
\[
\int_0^1
\tr
\sum_j g_j(t)g_j^*(t)
\,dt
= \sum_j
\|g_j^*\|_{L^2([0, 1); S_2)}^2
= \sum_j
\|g_j\|_{L^2([0, 1); S_2)}^2
\le 1.
\]
This gives the second inequality in line~\eqref{eq:SplitNorms}
with~$C=2$.

Apply equation~\eqref{eq:Projections} with~$M = A_jM_j$ to
rewrite~$a_j$ as~$\langle g, Q_jA_jh\rangle_p$. Let~$P_j$ be the orthogonal projection with range~$M_j$.
Since~$A_j$ is unitary,
\[
A_jP_j =
Q_jA_j 
\quad\textnormal{and}\quad
a_j = \langle g, A_jP_jh\rangle_p.
\]
Recall that~$g \perp_p A_{j+1}M_j$ for all~$j$. So~$\langle g, A_jP_{j-1}h\rangle_p = 0$ for all~$j \ge 1$. Make this true when~$j=0$ by letting~$P_{-1} = 0$. Then
\[
a_j = \langle g, A_j(P_j - P_{j-1})h\rangle_p
\quad\text{for all~$j$.}
\]
Use
the method applied to~$(b_j^*)$ to confirm that~$\|(a_j)\|_{\mathcal{C}_E} \le 1$.
\end{proof}

\begin{proof}[Proof of Case~\ref{it:Complementary} of Theorem~\ref{th:Paley}]
Transfer the argument above as follows.
Replace~$L^2([0, 1); S_2)$ with~$L^2(\T; S_2)$, and use the partial inner product given in formula~\eqref{eq:PartialProduct}.
Factor a function~$f$ in the unit ball of~$L^1(\T;S_1)$ as~$h^*g$ where~$g$ and~$h$ belong to the unit ball of~$L^2(\T;S_2)$.
Let~$z$ be the function
mapping each number~$t$
in the interval~$(-\pi, \pi]$ to~$e^{it}$. Then
\begin{equation}
\label{eq:coefficients}
\hat f(n)= \langle g, z^nh\rangle_p
\quad\text{for all~$n$.}
\end{equation}
Let~$A_j$ be the unitary operator on~$L^2(\T; S_2)$ that multiplies each function by~$z^{k_j}$.
Then~$\hat f(k_j) = \langle g, A_jh\rangle$.

Let~$M_j$ be the 
closure of the
subspace of~$L^2(\T; S_2)$ spanned by the products~$z^n hb$ 
for which~$-k_j \le n < 0$ and~$b \in B(\Hilbert)$. 
Then the
inclusions~\eqref{eq:M'sNest} hold for these subspaces, as do the inclusions~\eqref{eq:AM'sNest} for their images~$A_jM_j$. 
Again,
\[
A_jh \in A_{j+1}M_{j+1}
\quad\textnormal{and}\quad
(A_jM_j)B(\Hilbert) \subset A_jM_j.
\]
Now~$A_{j+1}M_j$ is 
the closure of the subspace
spanned by the
products~$z^nhb$ in which~$b \in B(\Hilbert)$
and~$n \in [k_{j+1} - k_j, k_{j+1})$.
By strong lacunarity,
that interval 
is included in $(k_j, k_{j+1})$.
The gap hypothesis on~$\hat f(n)$ 
and formulas~\eqref{eq:coefficients} and~\eqref{eq:Module}
then 
imply
that~$g\perp_p A_{j+1}M_j$.

Define the projections~$Q_j$ and~$P_j$ as before, and split~$\langle g, A_jh\rangle_p$ in the same way. Estimate~$\|(b^*_j)\|_{\mathcal{C}_E}$ and~$\|(a_j)\|_{\mathcal{C}_E}$ as above.
\end{proof}

\begin{proof}[Proof of Case~\ref{it:Paley} of Theorem~\ref{th:Paley}]
Use the same factorization and the same operators~$A_j$ as in Case~\ref{it:Complementary},
but replace the subspaces~$M_j$
with the closures,~$L_j$ say,
in~$L^2(\T; S_2)$ of the span of the products~$z^n hb$
for which~$n < -k_j$ and~$b \in B(\Hilbert)$.
These subspaces nest in the opposite way,
that is
\[
L_0 \supset L_1  \supset \cdots \supset L_J \supset\cdots.
\]
In every case, the subspace~$A_jL_j$ is the same, namely the closure of the span of
the products~$z^n hb$
with~$n < 0$. Formula~\eqref{eq:coefficients} and the hypothesis that~$\hat f(n)$ for all~$n < 0$ now makes~$g \perp_p A_jL_j$ for all~$j$.
The
lacunarity hypothesis implies that~$A_jh \in A_{j+1}L_j$ for all~$j$, and that
\[
A_1L_0 \subset A_2L_1 \subset \cdots \subset A_{j+1}L_j \subset \cdots.
\]
Finally,~$(A_{j+1}L_j)B(\Hilbert) \subset A_{j+1}L_j$ in all cases.

Now
denote the orthogonal projection onto~$A_{j+1}L_j$ by~$Q_{j+1}$, and let~$Q_0 = 0$.
Then it is again true that
\begin{gather}
\langle g, A_jh\rangle_p
= \langle g, Q_{j+1}A_jh\rangle_p
= \langle Q_{j+1}g, A_jh\rangle_p
\notag\\
=
\langle Q_jg, A_jh\rangle_p
+
\langle(Q_{j+1}-Q_j)g, A_jh\rangle_p
= a_j + b_j
\quad\textnormal{say.}
\notag
\end{gather}
Estimate~$\|(b_j^*)\|_{\mathcal{C}_E}
$ as before.

Note that~$a_0 = 0$
since~$Q_0 = 0$.
When~$j > 0$, rewrite~$a_j$  as~$\langle g, Q_jA_jh\rangle_p$.
For those indices~$j$, let~$P_{j-1}$ be the orthogonal projection with range~$L_{j-1}$.
This time,
\[
Q_jA_j = A_jP_{j-1}
\quad\textnormal{and}\quad
a_j = \langle g, A_jP_{j-1}h\rangle_p.
\]
Now~$\langle g, A_jP_{j}h\rangle_p = 0$,
since~$g \perp_p A_{j}L_j$ for all~$j$. 
Write
\[
a_j = \langle g, A_j(P_{j-1} - P_{j})h\rangle_p.
\]
The
desired estimate~$\|(a_j)\|_{\mathcal{C}_E} \le 1$ 
follows as before.
\end{proof}

\begin{remark}
\label{rm:WeakFactor}
A weaker form of generic factorization suffices.
In Theorem~\ref{th:Khintchine}, for instance, it is enough to prove that
\[
\ltrpv(\hat f(w_j))
\rvert\mspace{-1.2mu}\rvert\mspace{-1.2mu}\rvert
\le 2
\quad\textnormal{when~$\|f\|_{L^1(D;S_1)} < 1$.}
\]
By the definition of Bochner integration,~$f$ can then be represented as the sum of a series
\[
\sum_{m=1}^\infty f_m
\]
where the terms~$f_m$ are simple functions,
and~$\sum_{m=1}^\infty \|f_m\|_{L^1(D; S_1)} \le 1$.
Generic factorization with simple factors is easy to check for
simple functions.
Use this to write~$f_m = h_m^*g_m$, 
where~$\sum_{m=1}^\infty \|g_m\|_{L^2(D; S_2)}^2 \le 1$
and the same is true for the factors~$h_m$. 

The sequences~$g = (g_m)$ and~$h = (h_m)$ are members of the unit ball of the Hilbert
space~$\ell^2(L^2(D; S_2))$;
their inner product is equal to~$\tr(\langle g, h\rangle_p)$ for the partial inner product given by
\[
\langle g, h\rangle_p
=
\sum_{m=1}^\infty \int_{D} h_m^*g_m.
\]
The methods of this section work in this setting with~$L^2(D; S_2)$ replaced
by~$\ell^2(L^2(D; S_2))$,
and~$A_j$
redefined to act on~$(h_m)$ by termwise multiplication.
\end{remark}

\begin{remark}
\label{rm:Factorization}
The subspaces~$L_j$ used in Case~\ref{it:Paley} are invariant under multiplication by~$\overline z$, and their adjoints are invariant under multiplication by~$z$. In the discussion of scalar-valued functions in \cite{Forelli}, 
it is observed that the latter subspaces must be simply invariant
if~$f \in H^1(\T)$, and that
one can apply the characterization of simply invariant subspaces of~$L^2(\T)$ to show that both factors~$h^*$ and~$g$  
can be chosen to belong to~$H^2(\T)$.
One proof~\cite{Muh}
of analytic factorization
in~$H^1(\T; S_1)$ uses those ideas and more.
\end{remark}

\begin{remark}
\label{rm:Refine}
The
generic factorization methods work equally well with the hypotheses in Theorem~\ref{th:Paley} weakened to only require that~$\hat f$ vanish on suitable smaller sets of integers. 
The same smaller sets
arise in analyses of the dual methods for scalar-valued functions. See~\cite[Sections~4 and~5]{FouMiss} for more details.
\end{remark}

\begin{remark}
\label{rm:Equivalence}
In this version of the paper, we worked with 
closed
subspaces~$M$
for which the inclusion~$MB(\Hilbert) \subset M$ holds. 
Of course, this inclusion is really an equality, 
since the identity operator belongs to~$B(\Hilbert)$. 
We showed in Section~\ref{sec:two-step} that
the inclusion 
implies 
for
the orthogonal projection~$Q$ onto~$M$ 
that
\begin{equation}
\label{eq:AdjointForPartial}
\langle F, QG\rangle_p = \langle QF, G\rangle_p
\quad\textnormal{for all~$F$ and~$G$.}
\end{equation}
%
The converse is also true.
Indeed,
using equation~\eqref{eq:AdjointForPartial} with~$F$ in~$M$
and~$G\perp M$ yields that~$\langle F, G\rangle_p = 0$ in that case.
It then follows from equation~\eqref{eq:Module}
that~$Fb \perp G$ for all~$b$ in~$B(\Hilbert)$ 
and all~$G$ in~$M^\perp$, 
that is~$Fb \in\left (M^\perp\right)^\perp = M$.

The 
error
in the first version of this paper was the use of property~\eqref{eq:AdjointForPartial}, in six places like equation~\eqref{eq:SelfAdjoint}
in this version, for subspaces that 
do not have that property.
\end{remark}

\bibliographystyle{amsplain}

\end{document}